\documentclass[a4paper,11pt]{article}

\usepackage{float}
\usepackage{color}
\usepackage{soul}
\usepackage{mathtools,amsmath,amsfonts}
\usepackage{amssymb,amsthm, graphicx,bm}
\usepackage{color,enumerate}
\usepackage{url}
\usepackage{bm}
\usepackage{enumerate}
\usepackage{url}
\usepackage[utf8]{inputenc}
\usepackage{lscape}
\usepackage[noend]{algpseudocode}
\usepackage{hyperref}
\usepackage[ruled,vlined,linesnumbered]{algorithm2e}
\usepackage[graphicx]{realboxes}

\def\N{\mathbb{N}}

\def\e{\mathrm{e}}

\def\F{\mathrm{F}}

\newcommand{\bN}{\mathbb{N}}

\newtheorem{theorem}{Theorem}
\newtheorem{definition}[theorem]{Definition}
\newtheorem{proposition}[theorem]{Proposition}
\newtheorem{corollary}[theorem]{Corollary}
\newtheorem{lemma}[theorem]{Lemma}

\theoremstyle{remark}
\newtheorem{example}[theorem]{Example}

\newtheorem{remark}[theorem]{Remark}

\title{Generalized strongly increasing semigroups}

\author{E.R. Garc\'{\i}a Barroso, J.I. Garc\'{\i}a-Garc\'{\i}a,\\ A. Vigneron-Tenorio
\footnote{
\noindent The first-named author was partially supported by the Spanish Project
   MTM2016-80659-P. The second and third-named author were partially supported by the Spanish Project MTM2017-84890-P and Junta de Andaluc\'{\i}a group FQM-366.}
}

\date{}

\begin{document}

\maketitle
\abstract{
In this work we present a new class of numerical semigroups called GSI-semigroups. We see the relations between them and others families of semigroups and we give explicitly their set of gaps.
Moreover, an algorithm to obtain all the GSI-semigroups up to a given Frobenius number is provided and the realization of positive integers as Frobenius numbers of GSI-semigroups is studied.
}

{\small
       2010 {\it Mathematics Subject Classification:} Primary 20M14; Secondary 14H20.

{\it Key words:} Generalized strongly increasing semigroup, strongly increasing semigroup, Frobenius number, singular analytic plane curve.
}

\section*{Introduction}
Let $\N=\{0,1,2,\ldots\}$ be the set of nonnegative integers. A numerical semigroup is a subset $S$ of $\N$ closed under addition, $0\in S$ and $\N\backslash S$, its gapset,  is finite.
The least not zero element in $S$ is called the multiplicity of $S$, we denote it by $m(S)$.
Given a nonempty subset $A=\{a_1,\dots,a_n\}$ of $\N$ we denote by $\langle A \rangle$ the smallest submonoid of $(\N,+)$ containing $A$; the submonoid $\langle A\rangle$ is equal to the set $\N a_1+\dots+\N a_n$.
The minimal system of generators of $S$ is the smallest subset of $S$ generating it, and its cardinality, denoted by $\e(S)$, is known as the embedding dimension of $S$.
It is well known (see Lemma 2.1 from \cite{semigrupos}) that $\langle  A \rangle$ is a numerical semigroup if and only if $\gcd (A)=1$. The cardinality of $\N\setminus S$ is called the genus of $S$ (denoted by $g(S)$) and its maximum is known as the Frobenius number of $S$ (denoted by $\F(S)$).

Numerical semigroups appear in several areas of mathematics and its theory is connected with Algebraic Geometry and Commutative Algebra (see \cite{barucci1997maximality}, \cite{MR713060}) as well as with Integer Optimization (see \cite{VBlanco}) and Number Theory (see \cite{assi2016numerical}).
It is common the study of families of numerical semigroups, for instance symmetric semigroups, irreducible semigroups and strongly increasing semigroups (see \cite{GB-P}, \cite{MR1981406}) or the study or characterization of invariants, for instance the Frobenius number, the set of gaps, the genus, etc. (see \cite{A-GS},   \cite{MR3836842}, \cite{MR1981406} and \cite{MR3648517}).

Inspired in \cite{A-GS} and \cite{GB-P},
and with the aim of study the sets of gaps of strongly increasing semigroups (shorted by SI-semigroups), we introduce the concept of generalized strongly increasing semigroup (shorted by GSI-semigroups). These numerical semigroups $\bar S=\langle v_0,\dots,v_{h},\gamma \rangle$ are the gluing of a semigroup $S$ with $\N$, we denote them by  $S \oplus_{d,\gamma}\N$, where $S=\langle v_0/d,\dots, v_{h}/d \rangle$,  $d=\gcd(v_0,\dots,v_{h})>1$ and $\gamma\in \N$ such that $\gamma>\max\{d\F(S),v_h\}$.

Our main result is Theorem \ref{th:huecos}, where we describe the set of gaps of GSI-semigroups.
Since every SI-semigroup is a GSI-semigroup, our description of the gaps is also valid for SI-semigroups.
Semigroups of values associated with plane branches are always
SI-semigroups and 
their sets of gaps
describe topological invariants of the curves (see \cite{GB-P}) which are used to classify singular analytic plane curves.
Due to the fact that the condition for being GSI-semigroup is straightforward to check from a given system of generators, it is easy to construct subfamilies of GSI-semigroups and thus of SI-semigroups.

In this work we also compare the class of GSI-semigroups with other families of numerical semigroups obtained as gluing of numerical semigroups. These are the classes of telescopic, free and complete intersection numerical semigroups. In \cite{A-GS}, it is constructed the set of complete intersection numerical semigroups with given Frobenius number, and some special subfamilies as free and telescopic numerical semigroups, and numerical semigroups associated with irreducible singular plane curves are studied.
As we pointed above, we prove that SI-semigroups are always GSI-semigroups. We also prove that GSI-semigroups are not included in the other three above-mentioned families.

Some algorithms for computing GSI-semigroups are provided in this work. One of them computes the set of GSI-semigroups up to a fixed Frobenius number. We prove that for any odd number $f$, there is at least a GSI-semigroup which Frobenius number equal to $f$. For even numbers, it does not always happen. Thus, GSI-semigroups with even Frobenius numbers are also studied, and we present an algorithm to check whether a GSI-semigroup with a given even Frobenius number.

This work is organized as follows.
In Section \ref{seccion1}, we introduce the GSI-semigroups and some of their properties. We prove that SI-semigroups are GSI-semigroups (see Corollary \ref{c:SI}). We also compare GSI-semigroups with another families such as free, telescopic and complete intersection numerical semigroups.
In Section \ref{seccion2}, the main result of this paper is presented (Theorem \ref{th:huecos}). This theorem gives us an explicit formula for the set of gaps of  GSI-semigroups.
We finish this work with Section \ref{seccion3}, where we give an algorithm for computing the set of GSI-semigroups up to a fixed Frobenius number and we show some pro\-per\-ties of Frobenius numbers of GSI-semigroups. In this last section, we also provide an algorithm to test whether there is a GSI-semigroup with given even Frobenius number.

\section{Generalized strongly increasing semigroups}\label{seccion1}
The gluing of $S=\langle v_0,\ldots ,v_h \rangle $ and $\mathbb N$ with respect to $d$ and $\gamma$ with $\gcd(d,\gamma)=1$ (see \cite[Chapter 8]{semigrupos}) is the numerical semigroup $\N dv_0+\dots+\N dv_h+\N\gamma$. We denote it by $S \oplus_{d,\gamma}\N$.

\begin{definition}
A numerical semigroup $\bar S$ is a generalized strongly increasing semigroup
whenever $\bar S$ is the gluing of a numerical semigroup $S=\langle v_0,\dots,v_h\rangle $ with respect to $d$ and $\gamma$ (that is, $\bar S=S \oplus_{d,\gamma}\N$), where $d\in\N\setminus \{0,1\}$ and $\gamma\in\N$ with $\gamma>\max\{d\F(S),dv_{h}\}$
(note that $d$ and $\gamma$ are coprimes).
\end{definition}

The first example of GSI-semigroups are numerical semigroups generated by two positive integers $\bar S=\langle a,b \rangle$ with $a<b$. For these semigroups, set $S=\N=\langle 1 \rangle$, $d=a$, and $\gamma=b$. Since $\F(S)=-1$, $\gamma=b> \max\{d\F(S),d\}=\{a\cdot (-1),a\}=a$.
From Sylvester (see \cite{sylvester}), we know that the Frobenius numbers of these semigroups are given by the formula $a\cdot b-a-b$. Hence, every odd natural number is realizable as Frobenius number of a GSI-semigroup.

Our definition of GSI-semigroups is inspired on the one of SI-semigroups. We remind you how they are defined (see \cite{Ba-GB-P} for further details).

A sequence of positive integers $(v_0,\ldots,v_h)$ is called a
characteristic sequence if it satisfies the following two
properties:
\begin{enumerate}
\item[\rm{(CS1)}] Put $e_k=\gcd (v_0,\ldots,v_k)$ for $0\leq k \leq  h$. Then $e_k<e_{k-1}$ for $1\leq k \leq  h$
and $e_h=1$.
\item[\hypertarget{cs2}{\rm{(CS2)}}] $e_{k-1}v_k<e_kv_{k+1}$ for $1\leq k \leq  h-1$.
\end{enumerate}

\noindent We put  $n_k= \frac{e_{k-1}}{e_k}$ for $1\leq k\leq h$. Therefore
$n_k>1$ for $1\leq k \leq  h$ and $n_{h}=e_{h-1}$. If $h=0$, the only
characteristic sequence is $(v_0)=(1)$. If $h=1$, the sequence
$(v_0,v_1)$ is a characteristic sequence if and only if  $\gcd
(v_0,v_1)=1$. Property \hyperlink{cs2}{\rm{(CS2)}} plays a role  if and only if
$h\geq 2$.

\begin{lemma}(\cite[Lemma 1.1]{Ba-GB-P})
\label{llll}
Let $(v_0,\ldots,v_h)$, $h\geq 2$ be a characteristic sequence. Then,
\begin{enumerate}
\item[\hbox{\rm (i)}] $v_1<\cdots<v_h$ and $v_0<v_2$.
\item[\hbox{\rm (ii)}] Let $v_1<v_0$. If $v_0\not\equiv 0$ \hbox{\rm (mod $v_1$) } then $(v_1,v_0,v_2,\ldots,v_h)$ is a cha\-rac\-teristic sequence. If  $v_0\equiv 0$ \hbox{\rm (mod $v_1$) } then $(v_1,v_2,\ldots,v_h)$ is a cha\-rac\-teristic sequence.
\end{enumerate}
\end{lemma}

We denote by $\langle v_0,\ldots,v_h \rangle $ the semigroup generated by the characteristic sequence $(v_0,\ldots,v_h)$. Observe that  $\langle v_0,\ldots,v_h \rangle $  is a numerical semigroup.
A semigroup $S\subseteq \bN$  is Strongly Increasing (SI-semigroup) if $S\neq \{0\}$ and it is generated by a characteristic sequence. Note that by Lemma \ref{llll}, we can assume that $v_0<\dots <v_h$.

\begin{theorem}\label{theorem_construction_StP}
    Let $\bar S$ be a numerical semigroup with $\e(\bar S)=h+1.$ Then, $\bar S$ is strongly increasing if and only if one of the two next conditions holds:
    \begin{enumerate}
        \item  $h=1,$ $\bar S=\N\oplus_{d,\gamma} \N=\langle d,\gamma \rangle$, where $d$ and $\gamma$  are two coprime integers.
        \item $h>1,$ $\bar S=S\oplus_{d,\gamma}\N$, where $S=\langle  v_0,\ldots, v_{h-1} \rangle$ is a strongly increasing semigroup with embedding dimension $h$ and
        $\gamma,d>1$ are two coprime integer numbers  such that $\gamma>d\gcd(v_0,\ldots, v_{h-2})v_{h-1}.$
    \end{enumerate}
\end{theorem}

\begin{proof}
The case $h=1$ is trivial by definition of characteristic sequences.

Assume $h>1$ and that $\bar S=\langle \bar v_0,\ldots ,\bar v_h \rangle $ is a strongly increasing numerical semigroup with embedding dimension strictly greater than $2.$
Let $\bar e_{i}=\gcd(\bar v_0,\ldots,\bar  v_{i})$ for $0\leq i\leq h$. Put $v_i=\frac{\bar v_i}{\bar e_{h-1}}$ for $0\leq i \leq h-1$.
Then, $(v_0,\ldots, v_{h-1})$ is a characteristic sequence.
Let $S=\langle v_0,\ldots, v_{h-1}\rangle $.
Since $\e(\bar S)=h+1$, then $\e(S)=h$.
Set $\gamma=\bar v_{h}$ and $d=\bar e_{h-1}$, we get $\bar S=S\oplus_{d,\gamma}\N$.
We have that $\gamma=\bar v_h >\bar  v_{h-1}=d v_{h-1}$, and since $\bar S$ is a SI-semigroup,
\begin{eqnarray*}
        \gamma=\bar v_h&>&\frac{\bar e_{h-2}}{\bar e_{h-1}}\bar v_{h-1}=\frac{\gcd(\bar v_0,\ldots,\bar v_{h-2})}{\bar e_{h-1}}\bar v_{h-1}\\
        &=&
        \frac{\gcd(\bar e_{h-1}v_0,\ldots,\bar e_{h-1}v_{h-2})}{\bar e_{h-1}}\bar e_{h-1}v_{h-1}\\
        &=&\bar e_{h-1}\gcd(v_0,\ldots, v_{h-2})v_{h-1}\\
        &=&d\gcd(v_0,\ldots, v_{h-2})v_{h-1}.
\end{eqnarray*}

Conversely, let $S=\langle v_0,\ldots ,v_{h-1}\rangle$ be a strongly increasing semigroup with embedding dimension $h,$ and $\gamma,d>1$ be two coprime integer numbers such that  $\gamma>d\gcd(v_0,\ldots,v_{h-2})v_{h-1}.$
Denote  $e_i=\gcd(v_0,\dots,v_{i})$ for $i=0,\dots,h-1$.
Take $\bar S=\langle \bar v_0,\dots,\bar v_h=\gamma \rangle$ the gluing semigroup $S\oplus_{d,\gamma}\N$ and define $\bar e_i=\gcd(\bar v_0,\dots,\bar v_{i})$ for $i=0,\dots,h$.
We have that $\bar e_0=de_0>\cdots >\bar e_{h-1}=d e_{h-1}=d>\bar e_h=\gcd(\gamma,d)=1$. Since  $e_{i-1} v_i<e_i v_{i+1}$, for $1\leq i\leq h-1$ then $ e_{i-1} v_i d^2< e_i v_{i+1}d^2$ and therefore $\bar e_{i-1} \bar v_i<\bar e_i \bar v_{i+1}$. By hypothesis $d e_{h-2} v_{h-1}<\gamma$, hence $d^2 e_{h-2} v_{h-1}<d \gamma$ and therefore $\bar e_{h-2} \bar v_{h-1}<\bar  e_{h-1} \gamma$. We conclude that $\bar S$ is a SI-semigroup.
\end{proof}

The following result give us a formula for the conductor (the Frobenius number plus $1$) of a SI-semigroup.

\begin{proposition}(\cite[Proposition 2.3 (4)]{GB-P}, \cite[Proposition 1.2]{Ba-GB-P})
	\label{conductor}
	Let $S=\langle v_0,\ldots,v_h \rangle$ be the semigroup generated by the characteristic sequence $(v_0,\ldots,v_h)$.
	The conductor of the  semigroup $S$ is
	\[
	c(S)=\sum_{i=1}^h(n_i-1)v_i-v_0+1.
	\]
	Moreover, the conductor of $S$ is an even number and its genus is $g(S)=\frac {c(S)} 2$.
\end{proposition}

By Proposition \ref{conductor}, we get
\begin{eqnarray*}
	\F(S)&=&\sum_{i=1}^h(n_i-1)v_i-v_0=\sum_{i=1}^hn_iv_i-\sum_{i=1}^hv_i-v_0\\
	&=&(n_1v_1-v_2)+(n_2v_2-v_3)+\cdots\\
    & & +(n_{h-1}v_{h-1}-v_h)+n_hv_h-v_1-v_0\\
	&\leq&-(h-1)+n_hv_h-v_1-v_0\\
	&=&e_{h-1}v_h-v_0-v_1-h+1< e_{h-1}v_h-v_0-v_1.
\end{eqnarray*}

Assume that $\bar S=\langle \bar v_0,\dots ,\bar v_h\rangle $ is a SI-semigroup satisfying that $\bar v_0<\dots <\bar v_h=\gamma$. Set $d=\bar e_{h-1}=\gcd(\bar v_0,\dots,\bar v_{h-1})$, $\gamma=\bar v_h$ and $S=\langle \bar v_0/d,\dots, \bar v_{h-1}/d\rangle$.
We have that
$d\gamma=\bar e_{h-1}\bar v_h>\bar e_{h-1}\bar v_h-\bar v_0-\bar v_1>\F(\bar S)=\sum_{i=1}^h(\bar n_i-1)\bar v_i-\bar v_0=d\F(S)+(\bar e_{h-1}-1)\bar v_h=d\F(S)+(d-1)\gamma$. Thus $\gamma>d\F(S)$.
Since $\bar S$ is a SI-semigroup, the property \hyperlink{cs2}{\rm{(CS2)}} if fulfilled. Using that the generators are ordered, we obtain that $\bar v_h<\bar e_{h-1}\bar v_h=d\bar v_h< \gamma$.

So we can state the following result.

\begin{corollary}\label{c:SI}
	Every SI-semigroup is a GSI-semigroup.
\end{corollary}

There are semigroups with similar definitions to SI and GSI semigroups. For example, telescopic, free and complete intersection.

Let $S=\langle v_0,\dots, v_h\rangle$. For $k\in \{0,\dots,h\} $, set $e_k=\gcd(v_0,\dots, v_{k-1})$ ($e_0=v_0$). We say that $S$ is free whenever it is equal to $\N$ or it is the gluing of a free with $\N$.
The semigroup $S$ is telescopic if it is free for the rearrangement $v_0<\dots <v_h$.
A semigroup is complete intersection if it is the gluing of two complete intersection numerical semigroups.
The above three definitions are from \cite{A-GS}.

It is easy to check that SI-semigroups are telescopic, telescopic are free semigroups and free semigroups are complete intersection.
In general, GSI-semigroups are neither strongly increasing nor  telescopic nor free nor complete intersection.
Clearly, $\langle 6,14,22,23\rangle = \langle 3,7,11\rangle\oplus_{2,23}\N$ and $23>\max\{2 \F(\langle 3,7,11\rangle ),2\cdot 11\}$. Thus this is a GSI-semigroup.
We define the functions \texttt{IsSIncreasingNumericalSemigroup} and \texttt{IsGSI} to check is a numerical semigroup is a SI-semigroup and a GSI-semigroup, respectively (the code of these functions is showed in Table \ref{codigo1}).

\begin{table}[hb]
\begin{verbatim}
# The input is a NumericalSemigroup S
# The function returns true if S is a SI-semigroup
IsStronglyIncreasing:=function(S)
    local k, lEs, i, smg;
    smg:=MinimalGeneratingSystemOfNumericalSemigroup(S);
    for i in [2..Length(smg)] do
    	if(Gcd(smg{[1..i-1]})<=Gcd(smg{[1..i]})) then
        	return false;
        fi;
    od;
    for k in [1..(Length(smg)-2)] do
        if (Gcd(smg{[1..k]})*smg[k+1]>=
                    Gcd(smg{[1..(k+1)]})*smg[k+2]) then
        	return false;
        fi;
    od;
    return true;
end;

# The input is a NumericalSemigroup S
# The function returns true if S is a GSI-semigroup
IsGeneralizedStronglyIncreasing:=function(S)
    local smg,d,gamma,aux,S1,fn1;
    smg:=MinimalGeneratingSystemOfNumericalSemigroup(S);
    gamma:=smg[Length(smg)];
    d:=Gcd(smg{[1..Length(smg)-1]});
    aux:=(1/d)*smg{[1..Length(smg)-1]};
    S1:=NumericalSemigroup(aux);
    fn1:=FrobeniusNumber(S1);
    if(gamma<=d*fn1) then
        return false;
    fi;
    if(gamma<=d*aux[Length(aux)]) then
        return false;
    fi;
    return true;
end;
\end{verbatim}
\caption{GAP code of functions to check if a numerical semigroup is SI and/or GSI.}\label{codigo1} 
\end{table}

Applying our functions and the functions \texttt{IsFreeNumericalSemigroup}, \texttt{IsTelescopicNumericalSemigroup} and \texttt{IsCompleteIntersection} of \cite{NumericalSgps1.2.0} to the semigroup $\langle 6,14,22,23\rangle$, we obtain the following outputs:
\begin{verbatim}
gap> IsFreeNumericalSemigroup(
        NumericalSemigroup(6,14,22,23));
false
gap> IsTelescopicNumericalSemigroup(
        NumericalSemigroup(6,14,22,23));
false
gap> IsCompleteIntersection(
        NumericalSemigroup(6,14,22,23));
false
gap> IsSIncreasing(NumericalSemigroup(6,14,22,23));
false
gap> IsGSI(NumericalSemigroup(6,14,22,23));
true
\end{verbatim}
From the results of the above computations, we conclude that the class of GSI-semigroups contains the class of SI-semigroups, but it is different to the classes of free, telescopic and complete intersection semigroups.

\section{Set of gaps of a GSI-semigroup}\label{seccion2}
We have seen that GSI-semigroups are easy to obtain from any numerical semigroup just gluing it with $\N$ with appropiate elements $d$ and $\gamma$. Hence these semigroups form a large family within the set of numerical semigroups. In this section, we deepen into their study by explicitly determining their set of gaps.	

Hereafter the notation $[a \;\mathrm{mod}\;n]$ for an integer $a$ and a natural number $n$ means the remainder of the division of $a$ by $n$, and $[a]_n$ denotes the coset of $a$ modulo $n$. For any two real numbers $a\leq b$ we denote by $[a,b]_{\bN}$ the set of natural numbers belonging to the real interval $[a,b]$. Put $\lfloor a\rfloor $ the integral part of the real number $a$.

\begin{theorem}\label{th:huecos}
Let $S=\langle v_0,\dots ,v_h\rangle$ be a numerical semigroup with $v_0<\dots <v_h$, $d\geq 2$ and $v_{h+1}$ be two natural coprime numbers such that $v_{h+1}>\max\{ d\F(S),dv_{h}\}$. Then the gaps of the GSI-semigroup $\bar S= S\oplus_{d,v_{h+1}}\N$ are
\begin{equation}\label{gaps}
    \begin{array}{lll}
    \N\setminus \bar S & = &\left\{1,\ldots ,dv_0-1 \right\}\cup \left\{ x\in (dv_0,v_{h+1})\cap \N {\;:\;} x\notin dS \right\} \cup \\
    &  & {\cal A}_{d} \cup  \bigcup_{\ell=1}^{d-2} {\cal B}_{d,\ell},
    \end{array}
\end{equation}
where
\[
{\cal B}_{d,\ell}=
\left\{v_{h+1}+[\ell v_{h+1}\;\mathrm{mod}\;d]+kd {\;:\;} 0\leq k\leq \left\lfloor\frac{ \ell v_{h+1}}{d}\right\rfloor-1\right\}
\]
and
\[
{\cal A}_{d}=\bigcup_{k=1}^{d-1}\big(d(\N\setminus S)+kv_{h+1}\big) \mbox{ \rm (${\cal A}_d=\emptyset$ when $S=\N$)}.
\]
Moreover (\ref{gaps}) is a partition of the gapset of $\bar S$ (we do not write ${\cal A}_d$ or ${\cal B}_{d,l}$ if they are empty).
\end{theorem}
\begin{proof}
It is clear that $\{1,\ldots ,dv_0-1 \}$ is included in $\N\setminus \bar S$.

Consider $x\in (dv_0,v_{h+1})\cap \N$ such that $x\notin dS$, and suppose that $x\in \bar S$. Since $x<v_{h+1}$ then there are $\lambda_{i}$ with $0\leq i\leq h$ such that $x=\lambda_{0}dv_{0}+\cdots+\lambda_{h}dv_{h}=d(\lambda_{0}v_{0}+\cdots+\lambda_{h}v_{h})\in dS$, which is a contradiction. Hence we conclude that $\{ x\in (dv_0,v_{h+1})\cap \N \;:\; x\notin dS \}\subseteq \N\setminus \bar S$.

Suppose that $S\neq \bN$. Fix $1\leq k \leq d-1$ and let $x\in d(\N\setminus S)+kv_{h+1}$. We get $x=d\alpha+kv_{h+1}$, for some $\alpha\in \N\setminus S$. Suppose that $x\in \bar S$. So, there exist $\alpha_1,\ldots ,\alpha_h,\beta\in \N$ such that $x=d\alpha+kv_{h+1}=d\alpha_1v_1+\cdots + d\alpha_hv_h+\beta v_{h+1}$ and then $(k-\beta)v_{h+1}= d(\alpha_1v_1+\cdots + \alpha_hv_h-\alpha)$. If $k=\beta$, the element $\alpha$ have to belong to $S$ which it is not possible. Moreover, since $d$ and $v_{h+1}$ are coprime, $d$ divides $k-\beta$. If $\beta>k$, $d\alpha= d(\alpha_1v_1+\cdots + \alpha_hv_h)+(\beta-k) v_{h+1}$ with $v_{h+1}\in S$, that is, $\alpha \in S$. Again, it is not possible. If we assume $k>\beta$ then $k-\beta \ge d$ and $k\ge d$. In any case, the set $d(\N\setminus S)+kv_{h+1}$ is included in $\N\setminus \bar S$ for any integer $k$ in $[1,d-1]_{\bN}$.

Let us prove now that ${\cal B}_{d,\ell}\subseteq \N\setminus \bar S$.
Suppose that $x=v_{h+1}+[\ell v_{h+1}\;\mathrm{mod}\;d]+kd\in \bar S$ for some $1\leq \ell \leq d-2$ and $0\leq k\leq\left \lfloor\frac{ \ell v_{h+1}}{d}\right \rfloor-1$.
Let $\alpha_{0},\alpha_{1},\ldots,\alpha_{h+1}\in \bN$ such that $x=v_{h+1}+[\ell v_{h+1}\;\mathrm{mod}\;d]+kd=\alpha_{0}dv_{0}+\cdots+\alpha_{h}dv_{h}+\alpha_{h+1}v_{h+1}$. Hence,
$(\alpha_{h+1}-1)v_{h+1}-[\ell v_{h+1}\;\mathrm{mod}\;d]=d(k-\alpha_0v_0-\cdots-\alpha_hv_h)$  and
$[(\alpha_{h+1}-1-\ell)v_{h+1}]_d=[0]_d$. Since $d$ and $v_{h+1}$ are coprime then $d$ divides $\alpha_{h+1}-1-\ell$.
But $\max {\cal B}_{d,\ell}=(\ell +1)v_{h+1}-d$ so we have $\alpha_{h+1}\in\{0,1,\ldots,\ell\}$, hence $-1-\ell\leq \alpha_{h+1}-1-\ell \leq -1$ or equivalently $1+\ell \geq -\alpha_{h+1}+1+\ell \geq 1$ and $-\alpha_{h+1}+1+\ell$ is a multiple of $d$ which is a contradiction since $\ell<d-1$.

Taking into account the reasoning done so far we have
\[{\cal H}=\{1,\ldots ,dv_0-1 \}\cup \{ x\in (dv_0,v_{h+1})\cap \N \mid x\notin dS \} \cup {\cal A}_{d} \cup  \bigcup_{\ell=1}^{d-2} {\cal B}_{d,\ell}\subseteq \N\setminus \bar S.\]

Let us prove that  ${\cal H}$  is a partition (we do not write ${\cal A}_d$ or ${\cal B}_{d,\ell}$ when they are the emptyset).

When ${\cal A}_{d}$ is a nonempty set, let ${\cal A}_{d,k}=d(\N\setminus S)+kv_{h+1}$ for a fix $1\leq k\leq d-1$. In this case, if ${\cal B}_{d,\ell}$ is nonempty  we have
\begin{equation}
\label{minmax}
\max {\cal B}_{d,\ell}<\min {\cal A}_{d,\ell+1}\;\;\hbox{\rm for } 1\leq \ell \leq d-2.
\end{equation}

Observe that
\begin{equation}
\label{congr}
 [x]_{d}=[kv_{h+1}]_{d}\; \hbox{\rm for any } x\in {\cal A}_{d,k}\;
\end{equation} \hbox{\rm and }
\begin{equation}
\label{congr1}
 [y]_{d}=[(\ell +1)v_{h+1}]_{d} \; \hbox{\rm for any } y\in {\cal B}_{d,\ell}.
 \end{equation}

Since $1\leq k,\ell<d$
we get that any two sets ${\cal A}_{d,k}$ and ${\cal A}_{d,k'}$ are disjoint for $k\neq k'$ and any two sets ${\cal B}_{d,\ell}$ and ${\cal B}_{d,\ell'}$ are also disjoint for $\ell \neq \ell'$. Moreover ${\cal A}_d$ and ${\cal B}_{d,\ell}$ are also disjoint for any $1\leq \ell \leq d-2$. Indeed, let $x\in {\cal A}_{d}\cap {\cal B}_{d,\ell}$ for some $1\leq \ell \leq d-2$. Hence, there is $k\in \{1,\ldots,d-1\}$ such that $x\in {\cal A}_{d,k}$ and by (\ref{congr}) and (\ref{congr1}), $[x]_{d}=[kv_{h+1}]_d=[(\ell +1)v_{h+1}]_d$. Given that $d$ and $v_{h+1}$ are coprime and $1\leq k,\ell <d$ we get $k=\ell +1$. So $x\in {\cal A}_{d,\ell+1}\cap {\cal B}_{d,\ell}$, which is a contradiction by inequality (\ref{minmax}).

In order to finish the proof we will show that there is not a gap of $\bar S$ outside ${\cal H}$.

First at all, observe that if $x\in \N\setminus \bar S$ and $x<v_{h+1}$, $x\in \{1,\ldots ,dv_0-1 \}\bigcup \{ x\in (dv_0,v_{h+1})\cap \N \mid x\notin dS \}$.\\

{ Claim 1:} if $x\in \N\setminus \bar S$ and $v_{h+1}<x$ then $[x]_{d}= [kv_{h+1}]_d$, for some $k\in \{1,\ldots,d-1\}$.

Indeed, if we suppose that $x=\lambda d$ for some $\lambda\in \N$, by hypothesis we get $d\F(S)<v_{h+1}<x=\lambda d$, in particular $\lambda>\F(S)$, so $x\in dS\subset \bar S$.
Since $[x]_{d}\neq [0]_d$ and $\gcd(d,v_{h+1})=1$ we get $[x]_{d}\in \{ [1]_d,\ldots ,[d-1]_d\}=\{ [kv_{h+1}]_d\;:\; 1\le k \le d-1\}$, that is, any $x\in \N\setminus \bar S$ with $v_{h+1}<x$ is congruent with $kv_{h+1}$ module $d$ for some integer $k\in \{1,\ldots,d-1\}$.\\

We distinguish two cases, depending on ${\cal A}_d$. First, we suppose that ${\cal A}_d\neq \emptyset$.\\

{ Claim 2:} The greatest gap of $\bar S$ which is congruent with $kv_{h+1}$ modulo $d$ is $\max {\cal A}_{d,k}$, for $1\leq k\leq d-1$.\\
Let $x\in \N\setminus \bar S$ with $[x]_{d}=[kv_{h+1}]_{d}$ and $x> \max {\cal A}_{d,k}$ then  $x=d\F(S)+kv_{h+1}+\lambda d$ for some non zero natural number $\lambda$. So $x=d(\F(S)+\lambda)+kv_{h+1}\in \bar S$ since $\F(S)+\lambda\in S$.\\

{ Claim 3:} There are not gaps of $\bar S$ congruent with $(\ell+1)v_{h+1}$ modulo $d$, between $\max {\cal B}_{d,\ell}$ and $\min A_{d,\ell+1}$.\\
Remember that $[\max {\cal B}_{d,\ell}]_d=[\min A_{d,\ell+1}]_d=[(\ell+1)v_{h+1}]$. Suppose that $x\in \N \setminus \bar S$ with $\max {\cal B}_{d,\ell}<x<\min A_{d,\ell+1}$ and $[x]_d=[(\ell+1)v_{h+1}]_d$. Since $\max {\cal B}_{d,\ell}=(\ell+1)v_{h+1}-d$ and $\min A_{d,\ell+1}=(\ell+1)v_{h+1}+d$ the only possibility for $x$ is $(\ell+1)v_{h+1}$ which is an element of $\bar S$.\\

By Claims 1 and 2  we deduce that for any $x\in \N\setminus \bar S$ with $v_{h+1}<x$  there exists an integer $k_{0}\in \{1,\ldots,d-1\}$ such that $[x]_{d}=[k_0v_{h+1}]_d$ and $v_{h+1}<x\le \max {\cal A}_{d,k_0}$. In particular there is an integer number $\lambda$ such that $x=k_0v_{h+1}+\lambda d$.  Hence if $x\in [\min {\cal A}_{d,k_0}, \max {\cal A}_{d,k_0}]_{\bN}$ then $x\in {\cal A}_{d,k_0}$. Indeed, in this case $\lambda\in \N$ and $\lambda\not\in S$, otherwise $x\in \bar S$.\\

By Claim 3, we can assume that if $v_{h+1}<x<\min {\cal A}_{d,k_0}$ then $v_{h+1}<x\le \max {\cal B}_{d,k_0-1}=k_0v_{h+1}-d$. \\

{ Claim 4:} The set of all the integers in $(v_{h+1},\max {\cal B}_{d,k_0-1}]$ congruent with $k_0v_{h+1}$ module $d$ is $ {\cal B}_{d,k_0-1}$.\\
By \eqref{congr1} we have $[\max {\cal B}_{d,k_0-1}]_{d}=[k_{0}v_{h+1}]_{d}$. Moreover
\[
\{\max {\cal B}_{d,k_0-1},\, \max {\cal B}_{d,k_0-1}-d, \max {\cal B}_{d,k_0-1}-2d\ldots ,\min {\cal B}_{d,k_0-1}\}={\cal B}_{d,k_0-1}
\]
 and  $\min {\cal B}_{d,k_0-1}-d= v_{h+1}+[(k_0-1)v_{h+1}\mod d]-d <v_{h+1}$.\\

Hence $x$ has to belong to ${\cal B}_{d,k_0-1}$ and we finish the proof for the case ${\cal A}_d\neq \emptyset$.

Suppose now that ${\cal A}_d=\emptyset$,  that is $S=\N$ and $\bar S$ is generated by $d$ and $v_{1}$ ($h=0$).

{ Claim 5:} If  ${\cal A}_d= \emptyset$ then $\max {\cal B}_{d,\ell}$ is the greatest gap of $\bar S$ which is congruent with $(\ell+1)v_{h+1}$ modulo $d$, for $1\leq \ell \leq d-2$.\\
Observe that $\max {\cal B}_{d,\ell}=(\ell +1)v_1-d$. For any natural number $x>\max {\cal B}_{d,\ell}$ with
$[x]_d=[(\ell +1)v_1]_d,$ there is $\alpha\in \bN\setminus\{0\}$ such that $x=(\ell +1)v_1-d+\alpha d=(\ell +1)v_1+(\alpha-1) d\in \bar S$.
\end{proof}

The above result provides us an explicit formula for the gaps except the elements of ${\cal A}_d$. We now give some examples of GSI-semigroups where the set ${\cal A}_d$ is easily known.

\begin{example}
    Let $S=\langle 2,7\rangle$. We have $\N\setminus S=\{1,3,5\}$ and $\F(S)=5$.  Take now $d=2$ and $\gamma=15$. Since $\gamma > \max\{2\cdot 5, 2\cdot 7\}$, the semigroup $S\oplus_{2,15}\N$ is a GSI-semigroup.
    The set ${\cal A}_2$ is equal to $2\{1,3,5\}+1 \cdot 15=\{17,21,25\}$ and therefore $\F(S\oplus_{2,15}\N)=25$.
\end{example}

\begin{example}
	Consider now the semigroup $S=\langle 5,6,7,8,9\rangle$, $d=3$ and $\gamma=31$.
	We have $\N\setminus S=\{1,2,3,4\}$ and $\F(S)=4$.
	Since $\gamma > \max\{3\cdot 4, 3\cdot 9\}$, the semigroup $S\oplus_{3,31}\N$ is a GSI-semigroup and ${\cal A}_3=(3\{1,2,3,4\}+1 \cdot 31)\cup (3\{1,2,3,4\}+2 \cdot 31)=\{34,37,40,43,65,68,71,74\}$. Thus, $\F(S\oplus_{3,31}\N)=74$.
\end{example}

\begin{corollary}
Let $S=\langle v_0,\ldots ,v_h\rangle$ be a semigroup, and $d,v_{h+1}\in \N$ two natural numbers such that $\bar S= \langle dv_0,\ldots ,dv_h,v_{h+1}\rangle$ is a GSI-semigroup. Then
 \[
 \F(\bar S)=\left\{\begin{array}{ll}\max {\cal A}_d& \hbox{\rm if } {\cal A}_d\neq \emptyset\\
\max {\cal B}_{d,d-2} &  \hbox{\rm otherwise, }
\end{array}\right.
\]
where $ {\cal A}_d$ and ${\cal B}_{d,d-2}$ are from (\ref{gaps}).
\end{corollary}
\begin{proof}
If  ${\cal A}_d\neq \emptyset$, then, by inequality (\ref{minmax}) , $\F(\bar S)=\max{\cal A}_d=d\F(S)+(d-1)v_{h+1}$. Otherwise, $S=\N$ and $\bar S$ is generated by $d$ and $v_{1}$ ($h=0$). So $\F(\bar S)=(d-1)(v_1-1)-1=\max{\cal B}_{d,d-2}.$
\end{proof}

From the proof of Theorem \ref{th:huecos}, we obtain the Frobenius number of a GSI-semigroup $S\oplus_{d,\gamma}\N$, which is equal to \begin{equation}\label{FrobGSI}
\F(S\oplus_{d,\gamma}\N)=d\F(S)+(d-1)\gamma.\end{equation}

\section{Algorithms for GSI-semigroups}\label{seccion3}
We finish this work with some algorithms for computing GSI-semigroups. These algorithms focus on computing the GSI-semigroups up to a given Frobenius number, and on checking whether there is at least one GSI-semigroup with a given even Frobenius number. For any odd number, there is a GSI-semigroup with this number as its Frobenius number, however, this does not happen for a given even number. Thus, in this section we dedicate a special study to GSI-semigroups with even Frobenius number.

Algorithm \ref{algoritmo1} computes the set of GSI-semigroups with Frobenius number least than or equal to a fixed nonnegative integer.
Note that in step 5 of the algorithm we use that $\F(\bar S)=d\F(S)+(d-1)\gamma$ and $\gamma >d\F(S)$ implies that $\F(\bar S)\geq d^2\F(S)$ where $\bar S=S\oplus_{d,\gamma}\N$.

Denote by $M(S)$ the largest element of the minimal system of generators of a numerical semigroup $S$.

\begin{algorithm}[h]
	\BlankLine
	\KwData{$f\in\N\setminus\{0\}$.}
	\KwResult{The set $\{ \bar S\mid \bar S\textrm{ is a GSI-semigroup with } \F(\bar S)\leq f \}$. }
	\BlankLine
	${\cal A}=\emptyset$\;
	\ForAll {$k\in\{-1\}\cup\{1,2,\dots,f\}$}
	{\label{loop1}
		$B=\{S\mid \F(S)=k\}$ \;
		\ForAll{$S\in B$ }
		{\label{loop2}
			$D_{S}=\{d\in\N\setminus\{0,1\}\mid d^2\F(S)\leq f \}$  \;
			$G_{d,S}=\{(d,\gamma) \in D_{S}\times \N \mid \gcd(\gamma,d)=1,\,\gamma> \max\{d \F(S),d M(S) \},\, d\F(S)+(d-1)\gamma\leq f \}$ \;
		}
		${\cal A}={\cal A}\cup \{ S\oplus_{d,\gamma}\N\mid (d,\gamma)\in G_{d,S}\}$\;
	}
	\Return ${\cal A}$ \;
	\caption{Computation of the set of GSI-semigroups with Frobenius number least than or equal to $f$.}\label{algoritmo1}
\end{algorithm}

\begin{remark}
If $A$ is a minimal system of generators of a numerical semigroup $S$ and $d\in\N\setminus \{0,1\}$, then $dA$ is a minimal system of generators of $dS=\{ds\mid s\in S\}\subset d\N$. Furthermore, if $\gamma\in\N\setminus\{1\}$ and $\gcd(d,\gamma)=1$, then $\gamma\not\in d\N\setminus\{0\}$. Thus, $\gamma\not\in dS$ and $dA\cup \{\gamma\}$ is a minimal system of generators of $\langle dA \cup \{\gamma\}\rangle $.
\end{remark}

We give in Table \ref{GSI_semig_hasta_15} all the GSI-semigroups with Frobenius number least than or equal to $15$.

\begin{table}[H]
\centering
\begin{tabular}{|c|c|}
\hline
Frobenius number & Set of GSI-semigroups \\\hline
$1$ & $\{\langle 2, 3\rangle\}$ \\\hline
$2$ & $\emptyset$\\\hline
$3$ & $\{\langle 2,5 \rangle\}$ \\\hline
$4$ & $\emptyset$\\\hline
$5$ & $\{\langle 2, 7\rangle, \langle 3, 4\rangle\}$ \\\hline
$6$ & $\emptyset$\\\hline
$7$ & $\{\langle 2, 9\rangle, \langle 3, 5\rangle\}$ \\\hline
$8$ & $\emptyset$\\\hline
$9$ & $\{\langle 2, 11\rangle, \langle 4, 6, 7 \rangle\}$ \\\hline
$10$ & $\emptyset$\\\hline
$11$ &  $\{\langle 2, 13\rangle, \langle 3, 7\rangle, \langle 4, 5\rangle, \langle 4, 6, 9\rangle\}$ \\\hline
$12$ & $\emptyset$\\\hline
$13$ &  $\{\langle 2, 15\rangle, \langle 3, 8\rangle, \langle 4, 6, 11\rangle\}$\\\hline
$14$ & $\emptyset$\\\hline
$15$ &  $\{\langle 2, 17\rangle, \langle 4, 6, 13\rangle, \langle 6, 8, 10, 11\rangle\}$\\
\hline
\end{tabular}\caption{Sets of GSI-semigroups with Frobenius number up to 15.}\label{GSI_semig_hasta_15}
\end{table}

Remember that every numerical semigroup generated by two elements is a GSI-semigroup. Hence, for any odd natural number there exists at least one GSI-semigroup with such Frobenius number.

From Table \ref{GSI_semig_hasta_15}, one might think that there are no GSI-semigrups with even Frobenius number. This is not so and we can check that $\langle 9,12,15,16 \rangle=\langle 3,4,5 \rangle \oplus_{3,16}\N$ is a GSI-semigroup and its Frobenius number is $38$,
\begin{verbatim}
gap> FrobeniusNumber(NumericalSemigroup(9,12,15,16));
38
gap> IsGSI(NumericalSemigroup(9,12,15,16));
true
\end{verbatim}
This is the first even integer that is realizable as the Frobenius number of a GSI-semigroup.
We explain this fact: 
we want to obtain an even number $f$ from the formula (\ref{FrobGSI}), $f=\F(S\oplus_{d,\gamma}\N )=d \F(S)+(d-1)\gamma$. Since $\gcd(d,\gamma)=1$, then $d$ has to be odd and $\F(S)$  even.
Thus, the lowest number $f$ is obtained for the numerical semigroup $S$ with the smallest even Frobenius number, the smallest odd number $d\ge 3$ and the smaller feasible integer $\gamma$, that is, $S=\langle 3,4,5\rangle$, $d=3$ and $\gamma=16$.
Thus, the GSI-semigroup with the minimum even Frobenius number is $\langle 3,4,5\rangle \oplus_{3,16}\N$.

Note that not every even number is obtained as the Frobenius number of a GSI-semigroup $\langle 3,4,5\rangle\oplus_{3,\gamma}\N$ for some $\gamma\ge 16$ with $\gcd(d,\gamma)=1$. In this way, we only obtain the values of the form $36+2k$ with $k\in \N$ and $k\not \equiv 0 \mod 3$ (if $\gamma=16+k$ for $k\in \N$, $\F(\langle 3,4,5\rangle \oplus_{3,\gamma}\N)= 38+2k $). The numbers of the form $42 + 6 k$, with $k\in \N$,  are not obtained (see Table \ref{tt1}).

\begin{table}[H]
\centering
\begin{tabular}{c|c|c|c|c|c|c|c|c}
$\gamma$     & $16$ & $17$ & $18$ & $19$ & $20$ & $21$ & $22$ & $\dots$\\\hline
$\F( \langle 3,4,5\rangle \oplus_{3,\gamma}\N )$     & $38$ & $40$ & * & $44$ & $46$ & * & $50$ & $\dots$
\end{tabular}\caption{Values of $\gamma $ such that $\gcd(3,\gamma)\neq 1$ are marked with *.}\label{tt1}
\end{table}

We now look for GSI-semigroups with Frobenius number of the form $42+6k$. Reasoning as above, we use again the semigroup $\langle 3,4,5\rangle$ and set now $d=5$. In this case, the smallest Frobenius number is 114, and it is given by the semigroup $\langle 3,4,5\rangle\oplus_{5,26}\N$. In general, for the semigroups $\langle 3,4,5\rangle\oplus_{5,\gamma}\N$, the formula of their Frobenius numbers is $5 \cdot 2+ 4\gamma$ with $\gamma\geq 26$ and $\gamma \not \equiv 0 \mod 5$ (see Table \ref{tt2}).
For $S=\langle 3,4,5\rangle$, we fill all the even Frobenius number $f\geq 114$, excepting if $f$ is of the form $f=10+4k$ with $k=15k'$. That is, $f$ cannot be a number of the form $f=10+60k'$ with $k'\in \N\setminus\{0,1\}$, for instance $130$ and $190$.

\begin{table}[H]
\centering
\begin{tabular}{c|c|c|c|c|c|c|c|c}
$\gamma$  & $26$ & $27$ & $28$ & $29$ & $30$ & $31$ & $32$ & $\dots$\\\hline
$\F( \langle 3,4,5\rangle \oplus_{5,\gamma}\N )$     & $114$ & $118$ & $122$ & $126$ & * & $134$ & $138$ & $\dots$
\end{tabular}\caption{Values $\gamma $ such that $\gcd(5,\gamma)\neq 1$ are marked with *.}\label{tt2}
\end{table}

The above procedures are useful to construct GSI-semigroups with even Frobenius numbers, but with them we cannot determine if a given even positive integer is realizable as the Frobenius number of a GSI-semigroup.

Fixed an even number $f$, we are interested in providing an algorithm to check if there exists at least one GSI-semigroup $S\oplus_{d,\gamma}\N$ such that $\F(S\oplus_{d,\gamma}\N)=f$.

Using that $\gamma$ has to be greater than or equal to $d\F(S)+1\ge 3\F(S)+1$ (recall that $\gamma>\max\{d\F(S),dM(S)\}$ and $d\ge 3$)
and from formula (\ref{FrobGSI}), we obtain that if $\F(S\oplus_{d,\gamma}\N)=f$, then $2\le \F(S)\le \lfloor \frac{f-2}{9}\rfloor$. 

Let $t\in \left[2,\frac{f-2}{9}\right]$ be the Frobenius number of $S$. Hence,  $f=dt+(d-1)\gamma\ge d^2t+d-1$ and $d\in\left[3, \left\lfloor\frac{-1+\sqrt{4ft+4t+1}}{2t}\right\rfloor\right]$. Therefore, for $t\in \left[2,\frac{f-2}{9}\right]$ and $d\in\left[3, \left\lfloor\frac{-1+\sqrt{4ft+4t+1}}{2t}\right\rfloor\right]$, $\gamma$ equals $\frac{f-dt}{d-1}$.

The next lemma follows from the previous considerations.

\begin{lemma}
Given an even number $f$, $S\oplus_{d,\gamma}\N$ is a GSI-semigroup with Frobenius number $f$ if and only if $\F(S)$ is an even number belonging to $\left[2,\frac{f-2}{9}\right]$, $d$ is an odd number verifying $$d\in\left[3, \left\lfloor\frac{-1+\sqrt{4f\F(S)+4\F(S)+1}}{2\F(S)}\right\rfloor\right],$$ and $\gamma=\frac{f-d\F(S)}{d-1}$ is an integer number such that $\gcd(\gamma,d)=1$ and $\gamma>\max\{d\F(S),dM(S)\}$.
\end{lemma}

We present a family formed by semigroups $S$ of even Frobenius number with $\F(S)\geq 10$ and such that $M(S)\leq \F(S)$.

\begin{proposition}\label{S_f}
For every even number $f\geq 10$, the numerical semigroup $S_f$ minimally generated by $A=\{f/2 -1,f/2 +2,f/2 +3,\dots,2(f/2 -1)-1,2(f/2 -1)+1\}$ has Frobenius number equal to $f$.
\end{proposition}

\begin{proof}
Since $2(f/2-1)=f-2\not \in A$, $(f/2-1)+(f/2+2)=f+1$ and $\gcd(A)=1$, the set $A$ is a minimal system of generators of $S_f$ and $f \not \in S_f$.

The elements
$f+1=(f/2-1)+(f/2+2)$,
$f+2=(f/2-1)+(f/2+3)$,
$\dots$,
$f+(f/2-4)=(f/2-1)+2(f/2-1)-1$,
$f+(f/2-3)=(f/2-1)+(f/2-1)+(f/2-1)$, $f+(f/2-2)=(f/2-1)+(f-1)$,
$f+(f/2-1)=(f/2+2)+(2(f/2-1)-1)$
are $f/2-1$ consecutive elements in $S_f$. Hence $\F(S_f)=f$.
\end{proof}

The numerical semigroups with Frobenius numbers $2$, $4$, $6$ and $8$ are the following:
\begin{equation}\label{f2}
    \{\langle 3,4,5 \rangle \},
\end{equation} 
\begin{equation}\label{f4}
    \{\langle 3,5,7 \rangle,\langle 5,6,7,8,9\rangle \},    
\end{equation}
\begin{equation}\label{f6}
\{\langle 4,5,7\rangle,\langle 4,7,9,10\rangle,\langle 5,7,8,9,11\rangle,\langle 7,8,9,10,11,12,13\rangle\},
\end{equation}
and
\begin{equation}\label{f8}
\begin{array}{c}
\{\langle  3, 7, 11 \rangle, \langle  3, 10, 11 \rangle, \langle 5, 6, 7, 9\rangle,
\langle 5, 6, 9, 13\rangle, \langle 5, 7, 9, 11, 13\rangle,\\
\langle 5, 9, 11, 12, 13\rangle, \langle 6, 7, 9, 10, 11\rangle,
\langle 6, 9, 10, 11, 13, 14\rangle, \\
\langle 7, 9, 10, 11, 12, 13, 15\rangle, \langle 9,\dots,17\rangle \},
\end{array}
\end{equation}
respectively.

The semigroups of the sets  (\ref{f2}), (\ref{f4}), (\ref{f6}) and (\ref{f8}) and the families of Proposition \ref{S_f} are the seeds to construct the integers that are realizable as Frobenius numbers of GSI-semigroups.
We propose Algorithm \ref{algoritmo2} to check if there exist GSI-semigroups with a given even Frobenius number.

\begin{algorithm}[H]
	\BlankLine
	\KwData{$f$ an even number.}
	\KwResult{If there exists, a GSI-semigroup with Frobenius number $f$.}
	\BlankLine

    \If{$f<38$}
        {\Return $\emptyset$.}

    ${\cal S}=\{S\text{ numerical semigroup}\mid \F(S)\in 2\N\cap [2,\min\{8,\lfloor \frac{f-2}{9}\rfloor\}] \}$\;

    \ForAll {$S\in {\cal S}, d\in \left[3,\left\lfloor\frac{-1+\sqrt{4f\F(S)+4\F(S)+1}}{2\F(S)}\right\rfloor\right]$ odd, and $\gamma=\frac{f-d\F(S)}{d-1}\in \N$}
        {
        \If{$\big((\gamma>\max \{d\F(S),dM(S)\}) \wedge (\gcd(d,\gamma)=1)\big)$}
           {
           \Return $S\oplus_{d,\gamma}\N$\;
           }
        }
	${\cal A}=\left\{t\in [10,\lfloor \frac{f-2}{9}\rfloor]\mid t \text{ even }\right\}$\;
	\While {${\cal A}\neq \emptyset$}
        {
        $t=\text{First}({\cal A})$\;
        ${\cal A}={\cal A}\setminus\{t\}$\;
    	${\cal B}=\{d\in \left[3,\left\lfloor\frac{-1+\sqrt{4ft+4t+1}}{2t}\right\rfloor\right]\mid d \text{ odd }\}$\;
            \While {${\cal B}\neq \emptyset$}
                {
                $d=\text{First}({\cal B})$\;
                ${\cal B}={\cal B}\setminus\{d\}$\;
                $\gamma=\frac{f-dt}{d-1}$\;
                \If{$\big((\gamma \in\N ) \wedge (\gamma>dt) \wedge (\gcd(d,\gamma)=1)\big)$}
                    {
                    \Return $S_t\oplus_{d,\gamma}\N$\;
                    }
                }
        }
	\Return ${\cal A}$ \;
	\caption{Computation of a GSI-semigroup with even Frobenius number $f$ (if possible).}\label{algoritmo2}
\end{algorithm}

Note that several steps of Algorithm \ref{algoritmo2} can be computed in parallel way. We now illustrate it with a couple of examples.

\begin{example}
Let $f=42$, since $\lfloor \frac{42-2}{9}\rfloor= 4$, by Algorithm \ref{algoritmo2}, only the numerical semigroups with Frobenius number $2$ and $4$ must be considered.

If $\F(S)=2$, then $d\in\{3, 5\}$, since the odd numbers of the set  $\left[3,\lfloor \frac{-1+\sqrt{505}}{4}\rfloor\right]_\N$ are $3$ and $5$. For $d=3$, we have that $\gamma = \frac{42-3\cdot 2}{3-1}=18$, but $\gcd(d,\gamma)=1$ so we do not obtain any GSI-semigroup with Frobenius number $42$ from $S$ with $\F(S)=2$ and $d=3$. For $d=5$, $\gamma = \frac {42-5\cdot 2}{5-1}=8\not > 8 = d\F(S)=5\cdot 2=10$, obtaining again no GSI-semigroups.

If $\F(S)=4$, then $d=3$, since $[3,3]=\{3\}$, which is odd.
We obtain that $\gamma=\frac {42-3\cdot 4}{3-1}=15$. By (\ref{f4}), for $\F(S)=4$, we have $M(S)\geq 7$. In this case
$15\not > \max \{d\F(S),d M(S)\}=\max\{3\cdot 4,3\cdot 7\}=21$.

Hence, there are no GSI-semigroups with Frobenius number $42$.
\end{example}

\begin{example}
Consider $f=4620$. Using the code in Table \ref{codigoPar}, we check that there are no GSI-semigroups of the form  $S\oplus_{d,\gamma}\N$, with $\F(S)\in \{2, 4, 6, 8\}$. 
Nevertheless, the number $4620$ is realizable as the Frobenius number of a GSI-semigroup: the Frobenius number of $S_{12}\oplus_{13,372}\N$, $S_{12}\oplus_{17,276}\N$ and $S_{12}\oplus_{19,244}\N$ is $4620$.

With the code below, we also obtain other examples of Frobenius numbers of GSI-semigroups that cannot be constructed from semigroups $S$ with $\F(S)\in\{2,4,6,8\}$.
\begin{verbatim}
gap> t:=30000; # Bound of the Frobenius numbers.
gap> s1:=Difference([2..(t-2)],
            Union(ListOfFrobeniusD(2,t/2,t),
                Union(ListOfFrobeniusD(4,t/2,t),
                    Union(ListOfFrobeniusD(6,t/2,t),
                        ListOfFrobeniusD(8,t/2,t))))
            );
gap> s2:=ListOfFrobeniusD(12,t/2,t);
gap> Print(Intersection(s1,s2));
[ 4620, 7980, 26460 ]
\end{verbatim}
The new Frobenius numbers are $7980$ and $26460$. Some GSI-semigroups with these Frobenius numbers are:  $S_{12}\oplus_{13,652}\N$ and $S_{12}\oplus_{17,486}\N$ for $7980$, and $S_{12}\oplus_{13,2192}\N$ and $S_{12}\oplus_{17,1641}\N$ for $26460$.

\begin{table}[hb]
\begin{verbatim}
ListOfFrobenius:=function(fS,d,bound)
    local f,listF,gamma,lowerBound;
    listF:=[];f:=0;
    lowerBound:=d*fS;
    if(fS=2) then lowerBound:=fS*5; fi;
    if(fS=4) then lowerBound:=fS*7; fi;
    if(fS=6) then lowerBound:=fS*7; fi;
    if(fS=8) then lowerBound:=fS*9; fi;
    for gamma in [(lowerBound+1)..(bound-1)] do
        if(GcdInt(gamma,d)=1) then
            f:=d*fS+(d-1)*gamma;
            if(f<bound) then Append(listF,[f]);
            fi;
        fi;
    od;
    return listF;
end;

ListOfFrobeniusD:=function(fS,boundD,bound)
    local listF,d;
    d:=3; listF:=[];
    for d in List([1..Int((boundD-1)/2)],k->2*k+1) do
        listF:=Union(listF,ListOfFrobenius(fS,d,bound));
    od;
    return listF;
end;
\end{verbatim}
\caption{GAP code of functions to obtain some GSI-semigroups with even Frobenius numbers.}\label{codigoPar}
\end{table}
\end{example}

\bibliographystyle{unsrt}

\medskip
\noindent
{\small Evelia Rosa García Barroso\\
Departamento de Matem\'aticas, Estadística e I.O. \\
Secci\'on de Matem\'aticas, Universidad de La Laguna\\
Apartado de Correos 456\\
38200 La Laguna, Tenerife, Spain\\
e-mail: ergarcia@ull.es}

\medskip

\noindent {\small Juan Ignacio García-García \\
Departamento de Matem\'aticas\\
Universidad de C\'adiz\\
 E-11510 Puerto Real, C\'adiz, Spain\\
e-mail: ignacio.garcia@uca.es}

\medskip

\noindent {\small Alberto Vigneron-Tenorio\\
Departamento de Matem\'aticas\\
Universidad de C\'adiz\\
E-11406 Jerez de la Frontera, C\'adiz, Spain\\
e-mail: alberto.vigneron@uca.es}

\end{document}